\documentclass[12pt, reqno]{amsart}

\usepackage{amsmath}
\usepackage{amsfonts}
\usepackage{amssymb}
\usepackage{times}

\newtheorem{thm}{Theorem}
\newtheorem{lem}[thm]{Lemma}
\newtheorem{prop}[thm]{Proposition}
\newtheorem{cor}[thm]{Corollary}

\newtheorem*{thma}{Theorem A}
\newtheorem*{thmb}{Theorem B}
\newtheorem*{thmc}{Theorem C}

\newcommand{\C}{{\mathbb C}}
\newcommand{\D}{{\mathbb D}}
\newcommand{\B}{{\mathbb B}}
\newcommand{\R}{{\mathbb R}}
\newcommand{\T}{{\mathbb T}}
\newcommand{\bn}{{\mathbb B}_n}
\newcommand{\bnr}{r\bn}
\newcommand{\cn}{\C^n}
\newcommand{\sn}{{\mathbb S}_n}
\newcommand{\ins}{\int_{\sn}}
\newcommand{\mb}{{\mathcal M}_\beta}
\newcommand{\ind}{\int_\D}
\newcommand{\Ind}{{\rm Ind}\,}
\newcommand{\inb}{\int_{\bn}}

\begin{document}

\title[Spectral Theory for Multiplication Operators]
{Spectral Theory of Multiplication Operators\\ on Hardy-Sobolev Spaces}

\author{Guangfu Cao}
\address{Cao: Department of Mathematics, South China Agricultural University, 
Guangzhou, Guangdong 510640, China}
\email{guangfucao@163.com}

\author{Li He}
\address{He: School of Mathematics and Information Science,
Guangzhou University, Guangzhou 510006, China}
\email{helichangsha1986@163.com}

\author{Kehe Zhu}
\address{Zhu: Department of Mathematics and Statistics, SUNY, Albany, NY 12222, USA}
\email{kzhu@albany.edu}

\begin{abstract}
For a pointwise multiplier $\varphi$ of the Hardy-Sobolev space $H^2_\beta$ on
the open unit ball $\bn$ in $\cn$, we study spectral properties of the multiplication 
operator $M_\varphi: H^2_\beta\to H^2_\beta$. In particular, we compute the 
spectrum and essential spectrum of $M_\varphi$ and develop the Fredholm theory 
for these operators.
\end{abstract}

\thanks{Zhu's research was supported by NNSF of China (Grant No. 11571217), the Project 
of International Science and Technology Cooperation Innovation Platform in Universities in 
Guangdong Province (Grant No. 2014KGJHZ007), and Shantou University (Grant No. NTF17009). 
Cao and He's research was supported by NNSF of China (Grant No. 11501136, 11671152).}



\maketitle

\section{Introduction}

Let $\bn$ be the open unit ball in $\cn$ and $H(\bn)$ be the space of all holomorphic
functions on $\bn$. For $f\in H(\bn)$ we use
$$Rf(z)=z_1\frac{\partial f}{\partial z_1}(z)+\cdots+z_n\frac{\partial f}{\partial z_n}(z)$$
to denote the radial derivative of $f$ at $z$. If
$$f(z)=\sum_{k=0}^\infty f_k(z)$$
is the homogeneous expansion of $f$, then it is easy to see that
$$Rf(z)=\sum_{k=0}^\infty kf_k(z)=\sum_{k=1}^\infty kf_k(z).$$
More generally, for any real $\beta$ and any $f\in H(\bn)$ with the homogeneous expansion 
above, we define
$$R^\beta f(z)=\sum_{k=1}^\infty k^\beta f_k(z)$$
and call it the radial derivative of $f$ of order $\beta$. 

It is clear that these fractional radial differential operators satisfy $R^\alpha 
R^\beta=R^{\alpha+\beta}$. When $\beta<0$, the effect of $R^\beta$ on $f$ is 
actually ``integration'' instead of ``diffferentiation''. For example, radial differentiation 
of order $-3$ is actually radial integration of order $3$.

For $\beta\in\R$ the Hardy-Sobolev space $H^2_\beta$ consists of all holomorphic 
functions $f$ on $\bn$ such that $R^\beta f$ belongs to the classical 
Hardy space $H^2$. It is clear that $H^2_\beta$ is a Hilbert space with the inner 
product
$$\langle f,g\rangle_\beta=f(0)\overline{g(0)}+\langle R^\beta f, R^\beta g\rangle_{H^2}.$$
The induced norm in $H^2_\beta$ is then given by
$$\|f\|^2_\beta=|f(0)|^2+\|R^\beta f\|^2_{H^2}.$$

The multiplier algebra of $H^2_\beta$, denoted by $\mb$, consists of all functions 
$\varphi\in H(\bn)$ such that $\varphi f\in H^2_\beta$ for every $f\in H^2_\beta$.
A standard application of the closed-graph theorem shows that every $\varphi\in\mb$
induces a bounded linear operator $M_\varphi: H^2_\beta\to H^2_\beta$.  The
purpose of this paper is to study the spectral properties of these multiplication operators.
Our main results are the following.

\begin{thma}
Suppose $\beta\in\R$ and $\varphi\in\mb$. Then the spectrum of 
$M_\varphi: H^2_\beta\to H^2_\beta$ is the closure of $\varphi(\bn)$ in the 
complex plane.
\end{thma}

Note that the theorems above and below may look like simple extensions of
known results on the Hardy and Bergman spaces to the setting of Hardy-Sobolev
spaces. This is far from the truth. In fact, these results are surprising when we realize
that, for general $\beta$, the norm of $M_\varphi: H^2_\beta\to H^2_\beta$ is usually 
much bigger than $\|\varphi\|_\infty$!

\begin{thmb}
Suppose $\beta\in\R$ and $\varphi\in\mb$. Then the essential spectrum 
of $M_\varphi:H^2_\beta\to H^2_\beta$ is given by
$$\sigma_e(M_\varphi)=\bigcap_{r\in(0,1)}\overline{\varphi(\bn-r\bn)},$$
where $r\bn=\{z\in\cn: |z|<r\}$.
\end{thmb}

\begin{thmc}
Suppose $\beta\in\R$ and $\varphi\in\mb$. Then $M_\varphi: H^2_\beta\to H^2_\beta$ 
is Fredholm if and only if there exist $r\in(0,1)$ and $\delta>0$ such that 
$|\varphi(z)|\ge\delta$ for all $z\in\bn-\bnr$. Moreover, when $M_\varphi$ is Fredholm, 
its Fredholm index is always $0$ for $n>1$ and is equal to minus the winding number of 
the mapping $e^{it}\mapsto\varphi(re^{it})$, where $r\in(0,1)$ is sufficiently close to $1$. 
\end{thmc}

Hardy-Sobolev spaces have been studied in various papers in the literature, including
\cite{AB1, AB2, BO, CO1, CO2, CZ, CV, OF}. However, several different names have 
appeared for such spaces. For example, they were called weighted Bergman spaces 
in \cite{ZZ}, they were called holomorphic Sobolev spaces in \cite{BeBu}, and they 
were called Besov-Sobolev spaces in \cite{VW}.

It is well known that the Hardy-Sobolev spaces $H^2_\beta$ include several important 
spaces as special cases: the Hardy space ($\beta=0$), the Bergman space ($\beta=-1/2$), 
the Dirichlet space ($\beta=n/2$), and the Drury-Arveson space ($\beta=(n-1)/2$).
Our main results are certainly well known in the case of Hardy and Bergman spaces.
However, these results are highly nontrivial in the general case. In particular, we mention
that our results are new for the case of the Drury-Arveson space in higher dimensions 
and the case of the Dirichlet space even in dimension $1$. In fact, these two cases are the
main motivation for our general theory here. See \cite{Ar, Dr} for the original introduction
of the Drury-Arveson space and \cite{BBF1, BBF2, FX1, FX2, RS, S} for some recent work
about operator theory and function theory for the Drury-Arveson space.

The difficulty in the general case stems from the fact that the underlying
space is defined by properties of a certain derivative. This makes some problems that are
obvious for Bergman and Hardy spaces very difficult for the general case. For example, in
the computation of spectrum, we need to show that if $\varphi$ is a multiplier of the
space $H^2_\beta$ and $\lambda$ is a constant such that $|\lambda-\varphi(z)|\ge
\delta$ for some positive $\delta$ and all $z\in\bn$, then the function
$1/(\lambda-\varphi)$ is also a multiplier of $H^2_\beta$. If the space is defined in terms of 
the integrability of the function itself (such as the Hardy space and the Bergman space), this 
desired property is obvious. However, if the space is defined in terms of the integrability of a 
certain fractional derivative of $f$, then the problem becomes challenging.

\section{Some characterizations of $H^2_\beta$}

Recall that $H^2$ is the space of holomorphic functions $f$ on $\bn$ such that
$$\|f\|^2_{H^2}=\sup_{0<r<1}\ins|f(r\zeta)|^2\,d\sigma(\zeta)<\infty,$$
where $d\sigma$ is the normalized Lebesgue measure on the unit sphere
$\sn=\partial\bn$. It is well known that functions $f\in H^2$ have radial limits
$$f(\zeta)=\lim_{r\to1^-}f(r\zeta)$$
for almost all $\zeta\in\sn$. Moreover, the radial limit function $f(\zeta)$ above belongs
to $L^2(\sn,d\sigma)$. The inner product in $H^2$ can then be written as
$$\langle f,g\rangle_0=\langle f,g\rangle_{H^2}=\ins f(\zeta)\overline{g(\zeta)}
\,d\sigma(\zeta),$$
and its induced norm on $H^2$ is given by
$$\|f\|^2_0=\|f\|^2_{H^2}=\ins|f(\zeta)|^2\,d\sigma(\zeta).$$

It is well known that a function $f\in H(\bn)$ belongs to $H^2$ if and only if
$$\inb|Rf(z)|^2(1-|z|^2)\,dv(z)<\infty,$$
where $dv$ is normalized volume measure on $\bn$. See \cite{Ru, ZZ, Z}. More 
generally, for any $t>-1$, we consider the weighted volume measure
$$dv_t(z)=c_t(1-|z|^2)^t\,dv(z),$$
where $c_t$ is a positive normalizing constant such that $v_t(\bn)=1$. The spaces
$$A^2_t=L^2(\bn,dv_t)\cap H(\bn)$$
are called weighted Bergman spaces (with standard weights).

We begin with the following well-known Hardy-Littlewood type theorem for weighted 
Bergman spaces.

\begin{lem}
Suppose $f\in H(\bn)$, $p>0$, $t>-1$, and $\beta$ is real. If $p\beta+t>-1$, then
$$\inb|f(z)|^p\,dv_t(z)<\infty$$
if and only if
$$\inb(1-|z|^2)^{p\beta}|R^\beta f(z)|^p\,dv_t(z)<\infty.$$
Moreover, the two integrals above are comparable when $f(0)=0$, namely, each 
one dominates the other by a positive constant multiple that is independent of $f$.
\label{1}
\end{lem}

\begin{proof}
The case $\beta\ge0$ is well known to experts in the field. See \cite{Gr, ZZ, Z}. In 
particular, it was shown in Theorem 4.2 of \cite{CZ} that the fractional derivative 
$R^\beta$ can be replaced by another fractional derivative $R^{s,\beta}$, and in 
Theorem 2.19 of \cite{Z}, our desired result here was proved in terms of $R^{s,\beta}$. 
See Theorem 14 of \cite{ZZ} as well. 

If $\beta<0$, we let $\alpha=-\beta>0$ and let $g=R^\beta f$. Then the condition
$$\inb(1-|z|^2)^{p\beta}|R^\beta f(z)|^p\,dv_t(z)<\infty$$
can be written as
$$\inb|g(z)|^p\,dv_{p\beta+t}(z)<\infty,$$
which, according to the nonnegative case in the previous paragraph, is equivalent to
$$\inb(1-|z|^2)^{p\alpha}|R^\alpha g(z)|^p\,dv_{p\beta+t}(z)<\infty,$$
or
$$\inb|f(z)|^p\,dv_t(z)<\infty.$$
This proves the desired result.
\end{proof}

The following result characterizes Hardy-Sobolev spaces in terms of integrability with 
respect to weighted volume measures.

\begin{prop}
Suppose $\beta\in\R$ and $f\in H(\bn)$. Then the following conditions are equivalent.
\begin{enumerate}
\item[(a)] $f\in H^2_\beta$.
\item[(b)] $R^{\beta+1}f\in A^2_1$.
\end{enumerate}
If $N$ is a nonnegative integer with $N>\beta$, then the conditions above are 
also equivalent to
\begin{enumerate}
\item[(c)] $R^Nf\in A^2_{2(N-\beta)-1}$.
\end{enumerate}
\label{2}
\end{prop}

\begin{proof}
The equivalence of (a) and (b) follows from the definition of $H^2_\beta$, the fact that 
a function $g\in H(\bn)$ belongs to $H^2$ if and only if $Rg\in A^2_1$, and the fact that
$RR^\beta=R^{\beta+1}$.

If $N$ is a nonnegative integer such that $\beta<N$, we have $2(N-\beta)-1>-1$.
It follows from the equivalence of (a) and (b) that $f\in H^2_\beta$ if and only if
$$\inb|R^{\beta+1}f(z)|^2(1-|z|^2)\,dv(z)<\infty.$$
By Lemma \ref{1}, the condition above is equivalent to
$$\inb|R^{N-\beta-1}R^{\beta+1}f(z)|^2(1-|z|^2)^{1+2(N-\beta-1)}\,dv(z)<\infty,$$
or
$$\inb|R^Nf(z)|^2(1-|z|^2)^{2(N-\beta)-1}\,dv(z)<\infty.$$
This proves the equivalence of (a) and (c).
\end{proof}

Write the integral condition above as
\begin{equation}
\inb(1-|z|^2)^{2N}|R^Nf(z)|^2(1-|z|^2)^{-(2\beta+1)}\,dv(z)<\infty.
\label{eq1}
\end{equation}
With the language of generalized Bergman spaces from \cite{ZZ}, the condition in
(\ref{eq1}) above tells us the space $H^2_\beta$ is the generalized Bergman space
$A^2_{-(2\beta+1)}$ that was defined in \cite{ZZ}. We single out a few special cases
that are of particular importance in complex analysis and functional analysis. 

First, if $\beta=0$, $H^2_\beta$ of course becomes the classical 
Hardy space $H^2$, whose reproducing kernel is
$$K(z,w)=\frac1{(1-\langle z,w\rangle)^n}.$$ 
If $-(2\beta+1)=0$, or $\beta=-1/2$, it follows from \cite{ZZ, Z} that $H^2_\beta$ becomes 
the ordinary (unweighted) Bergman space $A^2$, whose reproducing kernel is
$$K(z,w)=\frac1{(1-\langle z,w\rangle)^{n+1}}.$$
If $-(2\beta+1)=-n$, or $\beta=(n-1)/2$, it follows from \cite{ZZ} that $H^2_\beta$ is
the so-called Drury-Arveson space, whose reproducing kernel is
$$K(z,w)=\frac1{1-\langle z,w\rangle}.$$
Finally, if $-(2\beta+1)=-(n+1)$, or $\beta=n/2$, it follows from \cite{ZZ} that 
$H^2_\beta$ is the Dirichlet space of the unit ball, whose reproducing kernel is
$$1+\log\frac1{1-\langle z,w\rangle}.$$
Note that, in order to obtain the specific form of the various reproducing kernels
above, it may be necessary to adjust the inner product in $H^2_\beta$ to a slightly
different (but equivalent) form. Details are left to the interested reader.

\section{Multipliers of $H^2_\beta$}

Although our main focus here is not on characterizations of pointwise multipliers (which
is a notoriously difficult problem in general), we still want to look at a couple of special 
cases in which the multipliers are relatively easy to determine.

It is well known that every pointwise multiplier $\varphi$ of $H^2_\beta$ must be in
$H^\infty$, the space of all bounded holomorphic functions. Also, the size of the space
$H^2_\beta$ decreases as $\beta$ increases. Thus
$$H^\infty\subset H^2\subset H^2_\beta$$
for $\beta\le0$.

\begin{prop}
For $\beta\le0$ we have $\mb=H^\infty$.
\label{3}
\end{prop}

\begin{proof}
If $\beta=0$, $H^2_\beta$ is the Hardy space $H^2$, whose multiplier algebra is
of course $H^\infty$. If $\beta<0$, it follows from (\ref{eq1}) that $H^2_\beta$ is 
the weighted Bergman space $A^2_t$ with $t=-(2\beta+1)>-1$. The multiplier algebra 
for such a Bergman space is clearly $H^\infty$.
\end{proof}

\begin{prop}
If $\beta>n/2$, then every function in $H^2_\beta$ is continuous up to the boundary,
$H^2_\beta$ is an algebra, and $\mb=H^2_\beta$.
\label{4}
\end{prop}

\begin{proof}
Since each space $H^2_\beta$ contains the constant function $1$, we always have
$\mb\subset H^2_\beta$.

Suppose $\beta>n/2$ and $f\in H^2_\beta$. Then $R^\beta f\in H^2$, so there exists
a positive constant $C$ such that
$$|R^{\beta}f(z)|\le\frac C{(1-|z|^2)^{n/2}},\qquad z\in\bn.$$
See Theorem 4.17 of \cite{Z}. It follows that
$$(1-|z|^2)^{\frac n2}|R^{\frac n2}R^{\beta-\frac n2}f(z)|\le C,\qquad z\in\bn.$$
This shows that the function $R^{\beta-\frac n2}f(z)$ belongs to the Bloch space
(see \cite{ZZ, Z} for example). It is well known that the fractional integral of any positive
order of a Bloch function is in the ball algebra.  Thus the function
$$f(z)=R^{\frac n2-\beta}R^{\beta-\frac n2}f(z)$$
is continuous up to the boundary.

Finally, suppose $\beta>n/2$ and $f,g\in H^2_\beta$. Let $N$ be a positive integer
greater than $\beta$. By Proposition~\ref{2}, $R^Nf$ and $R^Ng$ both belong to 
$A^2_{2(N-\beta)-1}$. We proceed to show that $R^N(fg)$ also belongs to
$A^2_{2(N-\beta)-1}$, which will prove that $H^2_\beta$ is an algebra and
$\mb=H^2_\beta$.

If $N=1$, then the desired result follows from
$$R(fg)=fRg+gRf$$
and the fact that both $f$ and $g$ are in $H^\infty$.

If $N=2$, we have
$$R^2(fg)=fR^2g+2RfRg+gR^2f.$$
The first and third term on the right-hand side both belong to $A^2_{2(N-\beta)-1}$,
because $R^2f$ and $R^2g$ both belong to $A^2_{2(N-\beta)-1}$ and $f$ and $g$ are
both bounded. To estimate the middle term, let
$$I=\inb|Rf(z)Rg(z)|^2(1-|z|^2)^{2(N-\beta)-1}\,dv(z).$$
By Holder's inequality, we have $I^2\le I(f)I(g)$, where
$$I(f)=\inb|Rf(z)|^4(1-|z|^2)^{2(N-\beta)-1}\,dv(z).$$
By Lemma \ref{1}, there is a positive constant $C$ such that
$$I(f)\le C\inb|R^2f(z)|^4(1-|z|^2)^{4+2(N-\beta)-1}\,dv(z).$$
Since $f$ belongs to the Bloch space, the function $(1-|z|^2)^2R^2f(z)$ is bounded. 
Thus there exists another positive constant $C$ such that
$$I(f)\le C\inb|R^2f(z)|^2(1-|z|^2)^{2(N-\beta)-1}\,dv(z)<\infty.$$
Similarly, $I(g)<\infty$. Thus $I<\infty$.

The case of more general $N$ is proved in the same manner. More specifically, by
the binomial formula, we have
$$R^N(fg)=\sum_{k=0}^N\binom{N}{k}R^kfR^{N-k}g.$$
The two terms corresponding to $k=0$ and $k=N$ are disposed of easily, as both $f$
and $g$ are bounded. Fix $0<k<N$ and consider the integral
$$I=\inb|R^kf(z)R^{N-k}g(z)|^2(1-|z|^2)^{2(N-\beta)-1}\,dv(z).$$
For any $1<p<\infty$ with $1/p+1/q=1$, we use Holder's inequality to obtain
$$I\le I_1^{1/p}I_2^{1/q},$$
where
$$I_1=\inb|R^kf(z)|^{2p}(1-|z|^2)^{2(N-\beta)-1}\,dv(z),$$
and
$$I_2=\inb|R^{N-k}g(z)|^{2q}(1-|z|^2)^{2(N-\beta)-1}\,dv(z).$$
By Lemma~\ref{1}, there exists a positive constant $C$ such that
\begin{eqnarray*}
I_1&\le& C\inb|R^{N-k}R^kf(z)|^{2p}(1-|z|^2)^{2p(N-k)+2(N-\beta)-1}\,dv(z)\\
&=&C\inb|R^Nf(z)|^{2p}(1-|z|^2)^{2p(N-k)+2(N-\beta)-1}\,dv(z).
\end{eqnarray*}
Let us now choose $p=N/k$ and $q=N/(N-k)$. Write $2p=2+2(p-1)$, observe that
$$2p(N-k)=2(p-1)N,$$
and use the boundedness of the function $(1-|z|^2)^NR^Nf(z)$ (which follows from
$f\in H^\infty$ and the fact that $H^\infty$ is contained in the Bloch space). We obtain 
another positive constant $C$ such that
$$I_1\le C\inb|R^Nf(z)|^2(1-|z|^2)^{2(N-\beta)-1}\,dv(z)<\infty.$$
The proof for $I_2<\infty$ is the same. This completes the proof of the proposition.
\end{proof}

The determination of the multiplier algebra $\mb$ for $0<\beta\le n/2$ is a much
more challenging (open) problem. Our focus here is not to characterize the multiplier
algebra. Instead, we will assume that $\varphi\in\mb$ and consider the spectral
properties of the bounded operator $M_\varphi: H^2_\beta\to H^2_\beta$.

\section{A differentiation formula}

A key step in the computation of spectrum for holomorphic multiplication operators
on $H^2_\beta$ is the following: if $\varphi$ is a multiplier of $H^2_\beta$ and if
$|\varphi(z)|\ge\delta$ for some positive constant $\delta$ and all $z\in\bn$, then
$1/\varphi$ is a multiplier of $H^2_\beta$ as well. This is obvious in the case of Hardy
and Bergman spaces. But when the space $H^2_\beta$ has to be defined in terms of
derivatives, the desired result becomes highly nontrivial.

In this section we prove a general formula of differentiation that will be critical to our
spectral analysis of multiplication operators on $H^2_\beta$. The formula is clearly of
independent interest, but unfortunately, its proof is not easy.

\begin{thm}
Suppose $N$ is a positive integer and $I$ is an open interval. If $f$ and $g$ are functions
on $I$ that are differentiable up to order $N$. Then
\begin{equation}
\sum_{k=0}^N(-1)^k\binom{N}{k}g^kD^{N-1}(g^{N-k}f)=0,
\label{eq2}
\end{equation}
where
$$D^n=\frac{d^n}{dx^n},\qquad n\ge0,$$
is the $n$-th order differential operator. 
\label{5}
\end{thm}

\begin{proof}
For $n=1$ we simply write 
$$D=D^1=\frac{d}{dx}.$$
We will prove the result by mathematical induction on $N$.

When $N=1$, the desired formula takes the form $gf-gf=0$, which is trivial.

Now assume that the formula in (\ref{eq2}) holds for some positive integer $N$
and all functions $f$ and $g$ on $I$ that are differentiable up to order $N-1$. We
will show that the same formula also holds when $N$ is replaced by $N+1$. To this end,
we assume that $f$ and $g$ are arbitrary functions on $I$ that are differentiable up to
order $N$. We apply $D$ to (\ref{eq2}) to obtain
\begin{align}
\sum_{k=0}^N&(-1)^k\binom{N}{k}kg^{k-1}D(g)D^{N-1}(g^{N-k}f)\label{eq3}\\
&=-\sum_{k=0}^N(-1)^k\binom{N}{k}g^kD^N(g^{N-k}f).\nonumber
\end{align}
Since this formula holds for all $f$ and $g$ differentiable up to order $N$, we can
replace $f$ by $gf$ in (\ref{eq3}) to obtain
\begin{align}
\sum_{k=0}^N&(-1)^k\binom{N}{k}kg^{k-1}D(g)D^{N-1}(g^{N+1-k}f)\label{eq4}\\
&=-\sum_{k=0}^N(-1)^k\binom{N}{k}g^kD^N(g^{N+1-k}f).\nonumber
\end{align}

Multiply equation (\ref{eq2}) by $g$. We obtain
$$\sum_{k=0}^N(-1)^k\binom{N}{k}g^{k+1}D^{N-1}(g^{N-k}f)=0.$$
Applying $D$ to this equation, we get
\begin{align}
\sum_{k=0}^N&(-1)^k\binom{N}{k}(k+1)g^kD(g)D^{N-1}(g^{N-k}f)\label{eq5}\\
&+\sum_{k=0}^N(-1)^k\binom{N}{k}g^{k+1}D^N(g^{N-k}f)=0.\nonumber
\end{align}
It is elementary to check that the first term above is equal to 
\begin{align}
-\frac1N\sum_{k=0}^N&(-1)^k\binom{N}{k}kg^{k-1}D(g)D^{N-1}(g^{N+1-k}f)\label{eq6}\\
&+\frac{N+1}N g\sum_{k=0}^N(-1)^k\binom{N}{k}kg^{k-1}D(g)D^{N-1}(g^{N-k}f).\nonumber
\end{align}
In fact, if we denote by $S$ the sum of the two terms above, then
\begin{align*}
S&=-\frac1N\sum_{k=1}^{N}(-1)^k\binom{N}{k}kg^{k-1}D(g)D^{N-1}(g^{N+1-k}f)\\
&\qquad+\frac{N+1}N\sum_{k=0}^N(-1)^k\binom{N}{k}kg^kD(g)D^{N-1}(g^{N-k}f)\\
&=\frac1N\sum_{k=0}^{N-1}(-1)^k\binom{N}{k+1}(k+1)g^kD(g)D^{N-1}(g^{N-k}f)\\
&\qquad+\frac{N+1}N\sum_{k=0}^N(-1)^k\binom{N}{k}kg^kD(g)D^{N-1}(g^{N-k}f)\\
&=\sum_{k=0}^{N-1}(-1)^k\left[\frac{k+1}N\binom{N}{k+1}+
\frac{k(N+1)}{N}\binom{N}{k}\right]g^kD(g)D^{N-1}(g^{N-k}f)\\
&\qquad+(-1)^N(N+1)g^ND(g)D^{N-1}(f)\\
&=\sum_{k=0}^N(-1)^k\binom{N}{k}(k+1)g^kD(g)D^{N-1}(g^{N-k}f).
\end{align*}

Combining (\ref{eq5}) and (\ref{eq6}), we see that equation (\ref{eq5}) becomes
\begin{align*}
\frac1N\sum_{k=0}^N&(-1)^k\binom{N}{k}g^kD^N(g^{N+1-k}f)\\
&-\frac{N+1}N g\sum_{k=0}^N(-1)^k\binom{N}{k}g^kD^N(g^{N-k}f)\\
&+\sum_{k=0}^N(-1)^k\binom{N}{k}g^{k+1}D^N(g^{N-k}f)=0.
\end{align*}
Cancel the third term from part of the second term above. We obtain
\begin{align*}
\frac1N\sum_{k=0}^N&(-1)^k\binom{N}{k}g^kD^N(g^{N+1-k}f)\\
&-\frac1N\sum_{k=0}^N(-1)^k\binom{N}{k}g^{k+1}D^N(g^{N-k}f)=0.
\end{align*}
Multiply both sides by $N$ and shift the index of summation in the second term. Then
$$\sum_{k=0}^N(-1)^k\binom{N}{k}g^kD^N(g^{N+1-k}f)+\sum_{k=1}^{N+1}
(-1)^k\binom{N}{k-1}g^kD^N(g^{N+1-k}f)=0.$$
Rewrite this as
\begin{align*}
&D^N(g^{N+1}f)+(-1)^{N+1}g^{N+1}D^N(f)\\
&+\sum_{k=1}^N(-1)^k\left[\binom{N}{k}+\binom{N}{k-1}\right]g^kD^N(g^{N+1-k}f)=0,
\end{align*}
and observe that
$$\binom{N}{k}+\binom{N}{k-1}=\binom{N+1}{k}.$$
We conclude that
$$\sum_{k=0}^{N+1}(-1)^k\binom{N+1}{k}g^kD^N(g^{N+1-k}f)=0.$$
This completes the induction step and finishes the proof of the theorem.
\end{proof}

Once again, we have tried very hard to find a more simple proof for the seemingly
elementary formula above, but we have been unsuccessful. Also, we certainly
realize that the formula may have appeared in the literature before. But we searched 
and it appears to be new!

We restate Theorem~\ref{5} as follows, which will be more directly related to our
spectral analysis for multiplication operators later on.

\begin{cor}
Suppose $N$ is a positive integer, $f$ and $g$ are $N$ times differentiable on the open
interval $I$, and $g$ is non-vanishing on $I$. Then
$$D^N\left(\frac fg\right)=\frac{(-1)^{N}}{g^{N+1}}
\sum_{k=0}^N(-1)^k\binom{N+1}{k}g^kD^N(g^{N-k}f).$$
\label{6}
\end{cor}

\begin{proof}
By Theorem~\ref{5}, we have
$$D^N(f)=-\frac{(-1)^{N+1}}{g^{N+1}}\sum_{k=0}^N(-1)^k\binom{N+1}{k}
g^kD^N(g^{N+1-k}f).$$
Replace $f$ by $f/g$. We obtain the desired result.
\end{proof}

\section{The spectrum of $M_\varphi$}

In this section we determine the spectrum of $M_\varphi: H^2_\beta\to H^2_\beta$
for $\varphi\in\mb$ and $\beta\in\R$. Note that the case of Hardy and Bergman spaces
is well known; see \cite{Z2} for example. Some preliminary work about the spectrum
of multiplication operators on the Dirichlet space can be found in \cite{CH}. 

The key here is the following result about reciprocals of pointwise multipliers. See 
\cite{RS} for related partial results in the case of the Drury-Arveson space.

\begin{prop}
Suppose $\varphi\in\mb$ and $|\varphi(z)|\ge\delta$ for some positive constant
$\delta$ and all $z\in\bn$. Then $1/\varphi$ belongs to $\mb$ as well.
\label{7}
\end{prop}

\begin{proof}
Choose a positive integer $N$ such that $N>\beta$. By Proposition~\ref{2}, a function
$f\in H(\bn)$ belongs to $H^2_\beta$ if and only if $R^N(f)\in A^2_{2(N-\beta)-1}$.
Since $2(N-\beta)-1>-1$, the multiplier algebra of $A^2_{2(N-\beta)-1}$ is $H^\infty$.
 
The radial derivative $R$ obeys the same differentiation rules as the ordinary derivative
$D$. Thus by Corollary~\ref{6}, we have
\begin{equation}
R^N\left(\frac f\varphi\right)=\frac{(-1)^{N}}{\varphi^{N+1}}
\sum_{k=0}^N(-1)^k\binom{N+1}{k}\varphi^kR^N(\varphi^{N-k}f)
\label{eq7}
\end{equation}
for all $f\in H^2_\beta$. Since $\varphi\in\mb$, each function $\varphi^{N-k}f$ belongs
to $H^2_\beta$, or equivalently, $R^N(\varphi^{N-k}f)\in A^2_{2(N-\beta)-1}$. Since
$\varphi$ and $1/\varphi$ are both bounded holomorphic functions on $\bn$, it
follows from (\ref{eq7}) that $R^N(f/\varphi)\in A^2_{2(N-\beta)-1}$, or equivalently,
$f/\varphi\in H^2_\beta$ whenever $f\in H^2_\beta$.
\end{proof}

We can now prove the first main result of the paper.

\begin{thm}
For any real $\beta$ and any $\varphi\in\mb$ the spectrum of
the operator $M_\varphi: H^2_\beta\to H^2_\beta$ is given by
$\sigma(M_\varphi)=\overline{\varphi(\bn)}$, which is the closure of the range of
$\varphi$ in the complex plane.
\label{8}
\end{thm}

\begin{proof}
It is easy to see that the point-evaluation at any $a\in\bn$ is a bounded linear
functional on $H^2_\beta$ (by Taylor expansion for example). Thus each $H^2_\beta$
is a reproducing kernel Hilbert space. For any $a\in\bn$ let $K(z,a)=K_a(z)$ denote 
the reproducing kernel of $H^2_\beta$ at $a$. Then
\begin{eqnarray*}
M_\varphi^*K_a(z)&=&\langle M_\varphi^*K_a, K_z\rangle\\
&=&\langle K_a, M_\varphi K_z\rangle=\langle K_a, \varphi K_z\rangle\\
&=&\overline{\varphi(a)}\overline{K_z(a)}=\overline{\varphi(a)}K_a(z).
\end{eqnarray*}
This shows that $M_\varphi^*K_a=\overline{\varphi(a)}K_a$, or $\overline{\varphi(a)}$
is an eigenvalue of $M_\varphi^*$. Thus $\overline{\varphi(a)}\in\sigma(M_\varphi^*)$, 
and so $\varphi(a)\in\sigma(M_\varphi)$. Since the spectrum of any bounded linear 
operator is closed, we must have $\overline{\varphi(\bn)}\subset\sigma(M_\varphi)$.

On the other hand, if $\lambda\in\C-\overline{\varphi(\bn)}$, then there exists a
positive number $\delta$ such that $|\lambda-\varphi(z)|\ge\delta$ for all $z\in\bn$.
Since $\varphi\in\mb$, we also have $\lambda-\varphi\in\mb$. By Proposition~\ref{7}, 
the function $\psi=1/(\lambda-\varphi)$ is also a pointwise multiplier of 
$H^2_\beta$. It is clear that 
$$M_\psi(\lambda I-M_\varphi)=(\lambda I-M_\varphi)M_\psi=I.$$
Thus $\lambda I-M_\varphi$ is invertible, or $\lambda\not\in\sigma(M_\varphi)$.
Combining this with what we proved in the previous paragraph, we conclude that
$\sigma(M_\varphi)=\overline{\varphi(\bn)}$.
\end{proof}

The theorem above is somewhat surprising, because it shows that the spectral
radius of the operator $M_\varphi: H^2_\beta\to H^2_\beta$ is always equal to
$\|\varphi\|_\infty$, while the norm of $M_\varphi$ can be strictly larger than
$\|\varphi\|_\infty$! For example, the norm of the operator $M_z$ on the classical
Dirichlet space on the unit disk is clearly greater than one.

More generally, in dimension one, the operator $M_z: H^2_\beta\to H^2_\beta$ is
a weighted shift whose spectrum is always the closed unit disk, although the weight
sequence can vary greatly as $\beta$ changes over $\R$.

\section{The Fredholm theory for $M_\varphi$}

In this section we compute the essential spectrum of $M_\varphi: H^2_\beta\to 
H^2_\beta$ when $\varphi$ is a multiplier of $H^2_\beta$. 

Recall that a bounded linear operator $T$ on $H^2_\beta$ is called a Fredholm operator 
if it has closed range, has finite dimensional kernel, and has finite dimensional co-kernel. 
When $T$ is Fredholm, the integer
$$\Ind(T)=\dim\ker(T)-\dim\ker(T^*)$$
is called the Fredholm index of $T$. 

The essential spectrum of a bounded linear operator
$T$ on $H^2_\beta$, denoted by $\sigma_e(T)$, is the set of complex numbers
$\lambda$ such that $\lambda I-T$ is not Fredholm. See \cite{D, Z2} for basic
information about Fredholm operators, the Fredholm index, and the essential spectrum. 

\begin{lem}
Let $T$ be a bounded linear operator on $H^2_\beta$. Then the following conditions
are equivalent.
\begin{enumerate}
\item[(a)] $T$ is not Fredholm.
\item[(b)] There exists a sequence $\{f_n\}$ of unit vectors in $H^2_\beta$ such that 
$f_n\to0$ weakly and $\|Tf_n\|\to0$ or $\|T^*f_n\|\to0$.
\end{enumerate}
\label{9}
\end{lem}

\begin{proof}
This is well known. See \cite{CH} for example.
\end{proof}

Our determination of the essential spectrum for $M_\varphi: H^2_\beta\to H^2_\beta$
depends on several different techniques that are valid in different situations. We begin 
with the high dimensional case.

\begin{lem}
Suppose $n>1$, $\beta$ is real, and $\varphi\in\mb$. Then
$$\sigma_e(M_\varphi)=\bigcap_{0<r<1}\overline{\varphi(\bn-r\bn)}
=\overline{\varphi(\bn)}=\sigma(M_\varphi).$$
\label{10}
\end{lem}

\begin{proof}
If $\lambda\in\varphi(\bn)$, then the function $\lambda-\varphi$ has a zero inside
$\bn$. Since $n>1$, the zero set of $\lambda-\varphi$ cannot be isolated. Thus 
$\lambda-\varphi$ has inifitely many distinct zeros inside $\bn$. If $\lambda-\varphi(a_k)
=0$ for infinitely many distinct points $a_k$ in $\bn$, then by the proof of 
Theorem~\ref{8}, $\ker(M^*_{\lambda-\varphi})$ contains all the kernel functions 
$K_{a_k}$, which span an infinite dimensional subspace in $H^2_\beta$. Thus $M^*_{\lambda-\varphi}$ is not Fredholm. This shows that $\varphi(\bn)\subset\sigma_e(M_\varphi)$. Taking the
closure, we obtain
$$\bigcap_{0<r<1}\overline{\varphi(\bn-r\bn)}\subset\overline{\varphi(\bn)}
\subset\sigma_e(M_\varphi)\subset\sigma(M_\varphi).$$

On the other hand, if
$$\lambda\not\in\bigcap_{0<r<1}\overline{\varphi(\bn-r\bn)},$$
then there exist $r\in(0,1)$ and $\delta>0$ such that $|\lambda-\varphi(z)|\ge\delta$
for all $z\in\bn-r\bn$. In particular, the function $\psi=1/(\lambda-\varphi)$ is
holomorphic on the shell $\bn-r\bn$. Since $n>1$, it follows from the Hartogs
extension theorem (see \cite{K}) and the maximum modulus principle that $\psi$ can 
be extended to a bounded holomorphic function on the whole unit ball $\bn$. Now 
the function $\psi(z)(\lambda-\varphi(z))$ is holomorphic on $\bn$ and equals $1$ on 
the shell $\bn-r\bn$. By the identity theorem, we have
$$\psi(z)(\lambda-\varphi(z))=1,\qquad z\in\bn.$$
This shows that $\lambda-\varphi$ is non-vanishing on $\bn$ and
$$\psi(z)=\frac1{\lambda-\varphi(z)},\qquad z\in\bn.$$
Since $\lambda-\varphi$ is bounded below on the shell $\bn-r\bn$ and is non-vanishing
on $r\bn$, it follows that $\lambda-\varphi$ is bounded below on the whole unit ball.
By Proposition~\ref{7}, the function $\psi$ is also a multiplier of $H^2_\beta$. Since
$$M_{\lambda-\varphi}M_\psi=M_\psi M_{\lambda-\varphi}=I,$$
we conclude that $\lambda I-M_\varphi=M_{\lambda-\varphi}$ is invertible, or
$\lambda\not\in\sigma(M_\varphi)$. This shows
$$\sigma(M_\varphi)\subset\bigcap_{0<r<1}\overline{\varphi(\bn-r\bn)}$$
and completes the proof of the lemma.
\end{proof}

It is easy to see that the result above does not hold for $n=1$. In fact, when $n=1$ and
$\varphi$ is continuous up to the boundary, then it is clear that, in general,
$$\overline{\varphi(\D)}=\varphi(\overline\D)\not=\varphi(\T)
=\bigcap_{0<r<1}\overline{\varphi(\D-r\D)},$$
where $\D=\B_1$ is the open unit disk and $\T$ is the unit circle in the complex plane.

\begin{lem}
Suppose $n=1$, $\beta\le1/2$, and $\varphi\in\mb$. Then
$$\sigma_e(M_\varphi)=\bigcap_{0<r<1}\overline{\varphi(\D-r\D)}.$$
\label{11}
\end{lem}

\begin{proof}
Let $K_a(z)=K(z,a)$ be the reproducing kernel of $H^2_\beta$ at $a\in\bn$. Let
$k_a=K_a/\|K_a\|$ be the normalized reproducing kernel of $H^2_\beta$ at $a\in\D$.
When $n=1$, it is easy to see that $k_a\to0$ weakly as $|a|\to1^-$ if and only if 
$\beta\le 1/2$.

First assume that
$$\lambda\in\bigcap_{0<r<1}\overline{\varphi(\D-r\D)}.$$
Then we can find a sequence $\{a_k\}\subset\D$ such that $|a_k|\to1$ and
$\varphi(a_k)\to\lambda$ as $k\to\infty$. Let 
$$T=\lambda I-M_\varphi=M_{\lambda-\varphi}: H^2_\beta\to H^2_\beta.$$
It follows  from the proof of Theorem~\ref{8} that
$$T^*k_{a_k}=(\overline\lambda-\overline{\varphi(a_k)})k_{a_k},\qquad k\ge1.$$
Therefore,
$$\lim_{k\to\infty}\|T^*k_{a_k}\|=\lim_{k\to\infty}|\varphi(a_k)-\lambda|=0.$$
By Lemma~\ref{9}, the operator $T$ is not Fredholm, or 
$\lambda\in\sigma_e(M_\varphi)$.

On the other hand, if
$$\lambda\not\in\bigcap_{0<r<1}\overline{\varphi(\D-r\D)},$$
then there exist $r\in(0,1)$ and $\delta>0$ such that $|\lambda-\varphi(z)|\ge\delta$
for all $z\in\D-r\D$. Let $a_1,\cdots, a_N$ denote the zeros of $\lambda-\varphi$ in 
$|z|\le r$, with multiple zeros repeated according to their multiplicities. Then
$$\lambda-\varphi(z)=\psi(z)p(z),$$
where
$$p(z)=(z-a_1)\cdots(z-a_N)$$
and $\psi$ is an invertible element of $H^\infty(\D)$. It follows from the proof of
Proposition~\ref{7} that $\psi\in\mb$ as well (which then also implies that 
$1/\psi\in\mb$), because the estimates there only involve membership in the
Bergman spaces $A^2_t$ with $t>-1$, and such membership of non-vanishing 
functions is determined by the behavior of the functions near the boundary. 

For each $k$ the operator $M_{z-a_k}$ is Fredholm on $H^2_\beta$. In fact, the range
of $M_{z-a_k}$ is the closed subspace 
$$I_k=(z-a_k)H^2_\beta=\{f\in H^2_\beta: f(a_k)=0\},$$
the kernel of $M_{z-a_k}$ is trivial (the operator is one-to-one), and the co-kernel of 
$M_{z-a_k}$ is equal to the one dimensional space spanned by $K_{a_k}$. Since 
$M_\psi$ is invertible on $H^2_\beta$ and $M_p=M_{z-a_1}\cdots M_{z-a_N}$ is 
Fredholm (the product of Fredholm operators is still Fredholm), we conclude that
$$\lambda I-M_\varphi=M_{\lambda-\varphi}=M_\psi M_p$$
is Fredholm. This completes the proof of the lemma.
\end{proof}

It remains for us to tackle the case $n=1$ and $\beta>1/2$. We need to come up with
a new proof, because for $\beta>1/2$, the normalized reproducing kernels $k_a$ no 
longer converge to $0$ weakly as $|a|\to1^-$.  Our new proof for this case will
be based on the notion of peak functions.

Fix some $\zeta\in\T$ and consider the functions
$$f_k(z)=\left(\frac{1+\overline\zeta z}2\right)^k,\qquad k=1,2,3,\cdots.$$
The function $f(z)=(1+\overline\zeta z)/2$ is traditionally called the peak function at
the boundary point $\zeta$.

\begin{lem}
There exists a positive constant $c$ such that 
$$\|f_k\|^2\ge c(k+1)^{2\beta-1}$$ 
for all $k\ge1$, where the norm is taken in $H^2_\beta$.
\label{12}
\end{lem}

\begin{proof}
Let $N$ be the smallest positive integer greater than $\beta$. Let us first consider 
the integrals
$$I_k=\ind|f_k(z)|^2\,dA_{2(N-\beta)-1}(z),\qquad k\ge1.$$
By the binomial formula and Stirling's formula,
\begin{eqnarray*}
I_k&=&\frac1{4^k}\sum_{j=0}^k\binom{k}{j}^2\ind|z|^{2j}(1-|z|^2)^{2(N-\beta)-1}\,dA(z)\\
&\sim&\frac1{4^k}\sum_{j=0}^k\binom{k}{j}^2\int_0^1r^j(1-r)^{2(N-\beta)-1}\,dr\\
&\sim&\frac1{4^k}\sum_{j=0}^k\binom{k}{j}^2\frac1{(j+1)^{2(N-\beta)}}\\
&\ge&\frac1{4^k(k+1)^{2(N-\beta)}}\sum_{j=0}^k\binom{k}{j}^2.
\end{eqnarray*}
By Cauchy-Schwarz inequality,
$$[2^k]^2=\left[\sum_{j=0}^k\binom{k}{j}\right]^2\le(k+1)\sum_{j=0}^k\binom{k}{j}^2.$$
It follows that
$$\sum_{j=0}^k\binom{k}{j}^2\ge\frac{4^k}{k+1},\qquad k\ge1.$$
Therefore, there exists a positive constant $c$, independent of $k$, such that
$$I_k\ge\frac c{(k+1)^{2(N-\beta)+1}}$$
for all $k\ge1$. If $k>N$, then by Proposition~\ref{2},
\begin{eqnarray*}
\|f_k\|^2&\sim&|f_k(0)|^2+\ind|R^Nf_k(z)|^2\,dA_{2(N-\beta)-1}(z)\\
&=&\frac1{4^k}+\ind|R^Nf_k(z)|^2\,dA_{2(N-\beta)-1}(z)\\
&\sim&\frac1{4^k}+(k+1)^{2N}I_{k-N}.
\end{eqnarray*}
Combining this with our earlier estimates for $I_k$, we obtain another positive
constant $c$ such that
$$\|f_k\|^2\ge\frac{c(k+1)^{2N}}{(k+1)^{2(N-\beta)+1}}=c(k+1)^{2\beta-1}$$
for all $k\ge1$.
\end{proof}

\begin{cor}
Let $g_k=f_k/\|f_k\|$ for $k\ge1$. Then $g_k\to0$ weakly in $H^2_\beta$ as 
$k\to\infty$.
\label{13}
\end{cor}

\begin{proof}
Each $g_k$ is a unit vector, and it follows from Lemma~\ref{12} above that $g_k(z)
\to0$ pointwise in $\D$ as $k\to\infty$. Thus for every $a\in\D$ we have
$\langle g_k, K_a\rangle\to0$ as $k\to\infty$. Since the set of finite linear combinations 
of kernel functions is dense in $H^2_\beta$, we conclude that for every $f\in H^2_\beta$
we have $\langle g_k,f\rangle\to0$ as $k\to\infty$. Consequently, $g_k\to0$ weakly in 
$H^2_\beta$ as $k\to\infty$.
\end{proof}

We are now ready to finish the last case in the determination of essential spectrum for
the multiplication operators $M_\varphi$ on $H^2_\beta$.

\begin{lem}
Suppose $n=1$, $\beta>1/2$, and $\varphi\in\mb$. Then
$$\sigma_e(M_\varphi)=\bigcap_{0<r<1}\overline{\varphi(\D-r\D)}.$$
\label{14}
\end{lem}

\begin{proof}
The proof for
$$\sigma_e(M_\varphi)\subset\bigcap_{0<r<1}\overline{\varphi(\D-r\D)}$$
is the same as the case $\beta\le1/2$.

On the other hand, if
$$\lambda\in\bigcap_{0<r<1}\overline{\varphi(\D-r\D)},$$
then there exists a sequence $\{z_k\}$ in $\D$ such that $|z_k|\to1$ and
$\varphi(z_k)\to\lambda$ as $k\to\infty$. Since $\beta>1/2$, it follows from 
Proposition~\ref{4} that $\varphi$ belongs to the disk algebra. Going down to a 
subsequence if necessary, we may assume that $z_k\to\zeta$ for some $\zeta\in\T$. 
Thus $\varphi(\zeta)=\lambda$ for some boundary point $\zeta$.
Let $\psi(z)=\lambda-\varphi(z)$. Then $\psi\in\mb$ and $\psi(\zeta)=0$. We will show
that $M_\psi=\lambda I-M_\varphi$ cannot be Fredholm, or $\lambda\in
\sigma_e(M_\varphi) $.

By Lemma~\ref{9} and Proposition~\ref{2}, it suffices for us to show that the integrals
$$\ind|R^N(\psi g_k)(z)|^2\,dA_{2(N-\beta)-1}(z)$$
converge to $0$ as $k\to\infty$, where $N$ is the smallest positive integer greater
than $\beta$ and $\{g_k\}$ is the sequence defined in Corollary~\ref{13}.

Let us write
$$R^N(\psi g_k)=\sum_{j=0}^N\binom{N}{j}R^j\psi R^{N-j}g_k$$
and consider the integrals
\begin{eqnarray*}
I_{k,j}&=&\ind|R^j\psi(z)R^{N-j}g_k(z)|^2\,dA_{2(N-\beta)-1}(z)\\
&=&\frac1{\|f_k\|^2}\ind|R^j\psi(z)R^{N-j}f_k(z)|^2\,dA_{2(N-\beta)-1}(z)\\
&\lesssim&\frac{(k+1)^{2(N-j)}}{\|f_k\|^2}\ind|R^j\psi(z)f_{k-(N-j)}(z)|^2\,
dA_{2(N-\beta)-1}(z),
\end{eqnarray*}
where $k\ge1$ and $0\le j\le N$.  Since $N$ is fixed and we are considering the limit
as $k\to\infty$, we may assume that $k$ is much larger than $N$. In this case, the
denominator above can be estimated by Lemma~\ref{12}, namely, we can find a positive
constant $C$ such that
\begin{equation}
I_{k,j}\le C(k+1)^{2(N-\beta)-2j+1}\ind|R^j\psi(z) f_{k-(N-j)}(z)|^2\,dA_{2(N-\beta)-1}(z)
\label{eq8}
\end{equation}
for all $k$ and $j$. Our goal is to show that $I_{k,j}\to0$ for all $0\le j\le N$ as 
$k\to\infty$.

The case $j=0$ calls for special attention, and this is the case where we critically use 
the condition that $\psi(\zeta)=0$. Recall that
$$I_{k,0}=\ind|\psi(z)R^Ng_k(z)|^2\,dA_{2(N-\beta)-1}(z).$$
Given $\varepsilon>0$ we break the unit disk into two parts, $\D=D_1\cup D_2$, where
$$D_1=\{z\in\D:|z-\zeta|<\delta\},\qquad D_2=\{z\in\D:|z-\zeta|\ge\delta\},$$
and $\delta$ is chosen so that $|\psi(z)|<\varepsilon$ for $z\in D_1$. Then
\begin{eqnarray*}
I_{k,0}&=&\int_{D_1}|\psi(z) R^Ng_k|^2\,dA_{2(N-\beta)-1}(z)\\
&&\qquad+\int_{D_2}|\psi(z)R^Ng_k(z)|^2\,dA_{2(N-\beta)-1}(z)\\
&<&\varepsilon^2\ind|R^Ng_k(z)|^2\,dA_{2(N-\beta)-1}(z)\\
&&\qquad+\|\psi\|_\infty^2\int_{D_2}|R^Ng_k(z)|^2\,dA_{2(N-\beta)-1}(z).
\end{eqnarray*}
Since $\{g_k\}$ is a sequence of unit vectors in $H^2_\beta$, it follows from
Proposition~\ref{2} that there exists a positive constant $C_1$, independent of $k$, 
such that
$$\ind|R^Ng_k(z)|^2\,dA_{2(N-\beta)-1}(z)\le C_1,\qquad k\ge1.$$
Since
$$R^Ng_k(z)\sim\frac{k^N}{\|f_k\|}f_{k-N}(z),$$
it follows from Lemma~\ref{12} that $R^Ng_k(z)\to0$ uniformly on $D_2$ as $k\to\infty$.
This shows that $I_{k,0}\to0$ as $k\to\infty$.

The case $j\ge2$ (which forces $N\ge2$) is the simplest. In fact, since $0<N-\beta\le1$, 
we have
$$2(N-\beta)-2j+1\le-1$$
for $j\ge2$. It follows from (\ref{eq8}) that
$$I_{k,j}\le\frac C{k+1}\ind|R^j\psi|^2\,dA_{2(N-\beta)-1}(z)\to0,\quad k\to\infty,$$
because $\psi\in H^2_\beta$ together with Proposition~\ref{2} implies that
\begin{align*}
\ind|R^j\psi(z)|^2&\,dA_{2(N-\beta)-1}(z)\\
&\sim\ind|R^N\psi(z)|^2(1-|z|^2)^{2(N-j)+2(N-\beta)-1}\,dA(z)\\
&\le\ind|R^N\psi(z)|^2\,dA_{2(N-\beta)-1}(z)<\infty.
\end{align*}

If $j=1$ and $0<N-\beta<1/2$, then
$$2(N-\beta)-2j+1<0.$$
It follows from (\ref{eq8}) and the argument above again that $I_{k,1}\to0$ as $k\to\infty$. 

If $j=1$ and $0<N-\beta=1/2$, then
$$I_{k,1}\le C\ind|R\psi(z)f_{k+1-N}(z)|^2\,dA_{2(N-\beta)-1}(z)\to0, \quad k\to\infty,$$
by dominated convergence.

Finally, if $j=1$ and $N-\beta>1/2$ (which forces $N\ge2$), then 
$$2(N-\beta)-2j+1>0,$$
or $2(N-\beta)-1>0$. In this case,
we recall from the remarks following Theorem 2.1 of \cite{Z} that 
$R^N\psi\in A^2_{2(N-\beta)-1}$ implies
$$\lim_{|z|\to1}|R^N\psi(z)|(1-|z|^2)^{N-\beta+\frac12}=0,$$
which is equivalent to
$$\lim_{|z|\to1}|R^{N-1}\psi(z)|(1-|z|^2)^{N-\beta-\frac12}=0.$$
Since $N\ge2$, we have $N-1\ge1$. Thus
$$\lim_{|z|\to1}|R\psi(z)|^2(1-|z|^2)^{2(N-\beta)-1}=0.$$
Recall that $\beta<N\le\beta+1$, so $0<2(N-\beta)\le2$. It follows from (\ref{eq8}) that
$$I_{k,1}\le C(k+1)\ind|R\psi(z)f_{k+1-N}(z)|^2\,dA_{2(N-\beta)-1}(z).$$
Given any $\varepsilon>0$, we choose $\delta\in(0,1)$ such that
$$|R\psi(z)|^2(1-|z|^2)^{2(N-\beta)-1}<\varepsilon,\qquad \delta<|z|<1.$$
Then by the change of variables $w=(1+\overline\zeta z)/2$ we have
\begin{eqnarray*}
I(\delta)&=:&(k+1)\int_{\delta<|z|<1}|R\psi(z) f_{k+1-N}(z)|^2\,dA_{2(N-\beta)-1}(z)\\
&\le&\varepsilon(k+1)\ind|f_{k+1-N}(z)|^2\,dA(z)\\
&=&\varepsilon(k+1)\ind\left|\frac{1+\overline\zeta z}{2}\right|^{2(k+1-N)}\,dA(z)\\
&\le&4\varepsilon(k+1)\ind|w|^{2(k+1-N)}\,dA(w)\\
&=&4\pi\varepsilon(k+1)/(k+1-N).
\end{eqnarray*}
On the other hand, it follows from uniform convergence that
$$\lim_{k\to\infty}(k+1)\int_{|z|\le\delta}|R\psi(z)f_{k+1-N}(z)|^2\,dA_{2(N-\beta)-1}(z)
=0.$$
This shows that $I_{k,1}\to0$ as $k\to\infty$. The proof of the lemma is now complete.
\end{proof}

Combining the last few lemmas, we obtain the following result about the essential
spectrum of multiplication operators on $H^2_\beta$.

\begin{thm}
Suppose $\beta$ is real and $\varphi\in\mb$. Then we always have
$$\sigma_e(M_\varphi)=\bigcap_{0<r<1}\overline{\varphi(\bn-r\bn)}.$$
If $n>1$, then
$$\sigma_e(M_\varphi)=\bigcap_{0<r<1}\overline{\varphi(\bn-r\bn)}
=\overline{\varphi(\bn)}=\sigma(M_\varphi).$$
\label{15}
\end{thm}

As a consequence of the theorem above and its proof, we obtain the following index
formulas for $M_\varphi$.

\begin{thm}
Suppose $\beta$ is real and $\varphi\in\mb$. Then $M_\varphi: H^2_\beta\to 
H^2_\beta$ is Fredholm  if and only if there exist $r\in(0,1)$ and $\delta>0$ such that
$|\varphi(z)|\ge\delta$ for all $r\le|z|<1$. When $M_\varphi$ is Fredholm,
$\Ind(M_\varphi)=0$ for $n>1$, and for $n=1$, $\Ind(M_\varphi)$ is equal to the 
winding number of the mapping $e^{it}\mapsto\varphi(re^{it})$ from the unit circle 
into $\C-\{0\}$.
\label{16}
\end{thm}

\begin{proof}
The desired characterization of Fredholm multiplication operators $M_\varphi$ is a
direct consequence of Theorem~\ref{15}. If $n>1$ and $M_\varphi$ is Fredholm, it follows
from the proof of Lemma~\ref{10} that $M_\varphi$ is actually invertible, so 
$\Ind(M_\varphi)=0$.

If $n=1$ and $M_\varphi$ is Fredholm, it follows from the proof of Lemma~\ref{11} that
$\varphi=\psi p$, where both $\psi$ and $1/\psi$ are multipliers of $H^2_\beta$ and
$p$ is a polynomial. Thus $M_\psi$ is invertible on $H^2_\beta$ and $\Ind(M_\varphi)
=\Ind(M_p)$, which is equal to the number of zeros of $p$ inside $\D$, with multiple 
zeros counted according to multiplicity. This shows that $\Ind(M_\varphi)$ is equal to 
the winding number of $\varphi$ restricted to the circle $|z|=r$.
\end{proof}

If $n>1$ and $\varphi\in\mb$ is continuous up to the boundary, then it is clear that
$$\varphi(\overline{\bn})=\overline{\varphi(\bn)}=
\bigcap_{0<r<1}\overline{\varphi(\bn-r\bn)}=\varphi(\partial\bn).$$
This is certainly a purely high dimensional phenomenon.

\end{document}